\documentclass[a4paper,12pt]{amsart}

\usepackage{amssymb, amscd}
\usepackage{latexsym,epsfig}
\usepackage[all]{xy}
\usepackage{pst-func}
\usepackage{pst-math}
\usepackage{pst-all}
\usepackage{hyperref}
\usepackage{color}
\numberwithin{equation}{section}
\def\today{\ifcase\month\or Jan\or Febr\or  Mar\or  Apr\or May\or Jun\or  Jul\or Aug\or  Sep\or  Oct\or Nov\or  Dec\or\fi \space\number\day, \number\year}

\newcommand{\CC}{\mathbb C}
\newcommand{\EE}{\mathbb E}

\newcommand{\PP}{\mathbb P}

\newcommand{\ZZ}{\mathbb Z}

\newcommand{\Sym}{{\mathrm{Sym}}}

\def\vandaag{\number\day\space\ifcase\month\or
 januari\or februari\or  maart\or  april\or mei\or juni\or  juli\or
 augustus\or  september\or  oktober\or november\or  december\or\fi,
\number\year}
\def\today{\ifcase\month\or
 Jan\or Febr\or  Mar\or  Apr\or May\or Jun\or  Jul\or
 Aug\or  Sep\or  Oct\or Nov\or  Dec\or\fi
 \space\number\day, \number\year}

\newtheorem{theorem}{Theorem}
\newtheorem{lemma}[theorem]{Lemma}
\newtheorem{proposition}[theorem]{Proposition}

\newtheorem{definition-lemma}[theorem]{Definition-Lemma}

\theoremstyle{definition}

\theoremstyle{remark}
\newtheorem{remark}[theorem]{Remark}

\theoremstyle{conclusion}

\begin{document}
\title[]{Covariants of binary sextics and modular forms \\
of degree 2 with character}

\author{Fabien Cl\'ery}
\address{Department of Mathematical Sciences,
Loughborough University,
UK}
\email{cleryfabien@gmail.com}

\author{Carel Faber}
\address{Mathematisch Instituut, Universiteit Utrecht,
Postbus Box 80010,
3508 TA Utrecht,
The Netherlands}
\email{C.F.Faber@uu.nl}

\author{Gerard van der Geer}
\address{Korteweg-de Vries Instituut, Universiteit van
Amsterdam, Postbus 94248,
1090 GE  Amsterdam, The Netherlands}
\email{G.B.M.vanderGeer@uva.nl}

\begin{abstract}\
We use covariants of binary sextics to describe the structure of 
modules of scalar-valued or vector-valued Siegel modular forms  of degree $2$ with character, over the ring of scalar-valued 
Siegel modular forms of even weight. For a modular form defined by a 
covariant we express the order of vanishing along the locus of 
products of elliptic curves in terms of the covariant.
\end{abstract}

\maketitle

\let\thefootnote\relax\footnotetext{The research of the first author
was supported by the EPSRC grant EP/N031369/1.}
\begin{section}{Introduction}
In \cite{CFvdG} we describe a map from covariants of binary sextics to Siegel
modular forms of degree $2$. If $V$ denotes the standard $2$-dimensional 
representation of ${\rm GL}(2,{\CC})$ with basis $x_1,x_2$ we consider the space ${\rm Sym}^6(V)$
of binary sextics. A general element $f\in {\rm Sym}^6(V)$ will be written as
$$
f= \sum_{i=0}^6 a_i \binom{6}{i} x_1^{6-i} x_2 ^i \, .
$$
The group ${\rm GL}(2,{\CC})$ acts on ${\rm Sym}^6(V)$.
We denote by $\mathcal{C}$ the ring of covariants of binary sextics. 
A bihomogeneous covariant has a
bi-degree $(a,b)$, meaning that it can be seen as a homogeneous 
expression of degree $a$ in the
coefficients $a_i$ of $f$ and as a form of degree $b$ in $x_1,x_2$;
such a covariant will be denoted by $C_{a,b}$.
The map from covariants to Siegel modular forms defined in \cite{CFvdG}
is a map
$$
\nu: \mathcal{C} \to M_{\chi_{10}}\, ,
$$
where $M$ is the ring of vector-valued modular forms of degree $2$ 
on $\Gamma_2={\rm Sp}(4,{\ZZ})$ and the subscript $\chi_{10}$ means
that Igusa's cusp form $\chi_{10}$ of weight $10$ is
inverted.
It sends the binary sextic 
$f$ to the meromorphic vector-valued modular form $\chi_{6,8}/\chi_{10}$
of weight $(6,-2)$, where $\chi_{6,8}$ is the unique holomorphic modular form
of weight $(6,8)$ (it is a cusp form). Using modular forms with character,
we can also write this as  $\chi_{6,3}/\chi_5$.
This map provides us with a very effective method
for constructing Siegel modular forms on $\Gamma_2$  
with or without character. 
We used it in \cite{CFvdG,C-vdG-Ch} to construct modular forms.

Since the image of a covariant under $\nu$ may be meromorphic on ${\mathcal A}_2$, with possible poles
along the locus ${\mathcal A}_{1,1}$ of abelian surfaces that are products of elliptic curves,
it is important to have a method to determine the order of vanishing of modular forms
obtained from covariants
along this locus. In this paper we give such a method. In our earlier papers 
\cite{CFvdG} and \cite{C-vdG-Ch} we relied on restriction of the corresponding
modular forms to the diagonal in the Siegel upper half space instead.

To exhibit the effectiveness of our method, we 
use it here to construct generators for certain modules of vector-valued 
Siegel modular forms of degree $2$. 

We denote by $M_{j,k}(\Gamma_2)$ (resp. $S_{j,k}(\Gamma_2)$)
the vector space of Siegel modular forms 
(resp.\ of cusp forms) of weight $(j,k)$ on $\Gamma_2$, 
that is,
the weight corresponds to the irreducible representation ${\rm Sym}^j({\rm St})\otimes
{\det{}}^k({\rm St})$ with ${\rm St}$ the standard representation of ${\rm GL}(2)$.
The group $\Gamma_2$ admits a character $\epsilon$ 
of order~$2$ and $\chi_5$, the square root of $\chi_{10}$, 
is a modular form of weight $5$ with this character. We refer to the last section
for a way to calculate the character. We denote the space
of modular forms (resp.\ of cusp forms) of weight $(j,k)$ with character
$\epsilon$ by $M_{j,k}(\Gamma_2,\epsilon)$ (resp.\ by $S_{j,k}(\Gamma_2,\epsilon)$).
 
Let $R=\oplus_{k \, {\rm even}} M_k(\Gamma_2)$ be the ring of scalar-valued Siegel modular
forms of degree $2$ of even weight. Igusa showed that it is a polynomial 
ring generated by $E_4,E_6,\chi_{10}$ and $\chi_{12}$.

We are interested in the structure of the $R$-modules
$$
\mathcal{M}_j^{\rm ev}(\Gamma_2,\epsilon) =\oplus_{k \, \text{even}}  M_{j,k}(\Gamma_2,\epsilon)
\quad \text{and}\quad
\mathcal{M}_j^{\rm odd}(\Gamma_2,\epsilon) =\oplus_{k\, \text{odd}} M_{j,k}(\Gamma_2,\epsilon)\, .
$$
The structure of the analogous modules for modular forms without character
$$
\mathcal{M}_j^{\rm ev}(\Gamma_2)=
\oplus_{k\, \text{even}}  M_{j,k}(\Gamma_2)  \quad {\rm and}\quad
\mathcal{M}_j^{\rm odd}(\Gamma_2)=
\oplus_{k\, \text{odd}}  M_{j,k}(\Gamma_2)
$$
is known for some values of $j$ by 
work of Satoh, Ibukiyama, van Dorp, Kiyuna, and Takemori, see \cite{Satoh,Ibu,vanDorp,Kiyuna,Takemori}.
The next table summarizes
the results.

\begin{footnotesize}
\smallskip
\vbox{
\bigskip\centerline{\def\quad{\hskip 0.6em\relax}
\def\quod{\hskip 0.5em\relax }
\vbox{\offinterlineskip
\hrule
\halign{&\vrule#&\strut\quod\hfil#\quad\cr
height2pt&\omit &&\omit  &&\omit &&\omit &&\omit &&\omit   &\cr
& $j$  && 2 && 4 && 6 && 8 && 10 &\cr
\noalign{\hrule}
& even && Satoh \cite{Satoh} && Ibukiyama \cite{Ibu} && Ibukiyama \cite{Ibu} && Kiyuna \cite{Kiyuna}  && Takemori
\cite{Takemori} &\cr
\noalign{\hrule}
& odd && Ibukiyama \cite{Ibu}  && Ibukiyama \cite{Ibu}  && van Dorp \cite{vanDorp} && Kiyuna \cite{Kiyuna} && Takemori 
\cite{Takemori} &\cr
} \hrule}
}}
\end{footnotesize} 

\smallskip

\noindent
The difficult part  is the construction of the generators
and the authors just mentioned used an array of methods to construct generators.
For example, Satoh used generalized Rankin-Cohen brackets, 
Ibukiyama used theta series for even unimodular lattices,
van Dorp used differential operators, and so on. Here we produce the generators 
we need by a uniform method via the covariants of binary sextics. 
We treat the cases $j=0, 2, 4, 6, 8, 10$ even and odd. In all these cases the module 
turns out to be a free $R$-module.

\noindent
{\bf Acknowledgement.} The authors thank the Max-Planck-Institut f\"ur
Mathematik in Bonn for the hospitality enjoyed while this work was done.

\end{section}
\begin{section}{The ring of covariants of binary sextics}
We recall some facts about the ring $\mathcal{C}$ of covariants of binary sextics.
For a description of 
$\mathcal{C}$ 
we refer to \cite{CFvdG,C-vdG-Ch}
and the classical literature mentioned there. 
The book of Grace and Young \cite[p.~156]{G-Y} gives 26 generators for this ring.
All these generators can be obtained as (repeated) so-called transvectants of 
the binary sextic $f$. The $k$th transvectant of two forms $g\in {\rm Sym}^m(V)$,
$h\in {\rm Sym}^n(V)$ is defined as 
$$
(g,h)_k=\frac{(m-k)!(n-k)!}{m!\,  n!}\sum_{j=0}^k (-1)^j
\binom{k}{j}
\frac{\partial^k g}{\partial x_1^{k-j}\partial x_2^j}
\frac{\partial^k h}{\partial x_1^{j}\partial x_2^{k-j}}
$$
and the index $k$ is usually omitted if $k=1$. 
If $g$ is a covariant of bi-degree $(a,m)$ and $h$ a covariant of bi-degree $(b,n)$,
then $(g,h)_k$ is a covariant of bi-degree $(a+b,m+n-2k)$
(cf.~\cite{Chi}).
The following table
summarizes the construction of the $26$ generators.

\begin{footnotesize}
\smallskip
\vbox{
\bigskip\centerline{\def\quad{\hskip 0.6em\relax}
\def\quod{\hskip 0.5em\relax }
\vbox{\offinterlineskip
\hrule
\halign{&\vrule#&\strut\quod\hfil#\quad\cr
height2pt&\omit &&\omit &&\omit &&\omit &&\omit &\cr
& 1 && $C_{1,6}=f$  && && && &\cr
\noalign{\hrule}
& 2  && $C_{2,0}=(f,f)_6$ && $C_{2,4}=(f,f)_4$  && $C_{2,8}=(f,f)_2$ && &\cr
\noalign{\hrule}
& 3  && $C_{3,2}=(f,C_{2,4})_4$ && $C_{3,6}=(f,C_{2,4})_2$ && $C_{3,8}=(f,C_{2,4})$ && $C_{3,12}=(f,C_{2,8})$ &\cr
\noalign{\hrule}
& 4  && $C_{4,0}=(C_{2,4},C_{2,4})_4$ && $C_{4,4}=(f,C_{3,2})_2$ &&  $C_{4,6}=(f,C_{3,2})$ && $C_{4,10}=(C_{2,8},C_{2,4})$ &\cr
\noalign{\hrule}
& 5  && $C_{5,2}=(C_{2,4},C_{3,2})_2$ && $C_{5,4}=(C_{2,4},C_{3,2})$ && $C_{5,8}=(C_{2,8},C_{3,2})$ && &\cr
\noalign{\hrule}
& 6  && $C_{6,0}=(C_{3,2},C_{3,2})_2$ && $C_{6,6}^{(1)}=(C_{3,6},C_{3,2})$ && $C_{6,6}^{(2)}=(C_{3,8},C_{3,2})_2$ && &\cr
\noalign{\hrule}
& 7  && $C_{7,2}=(f,C_{3,2}^2)_4$ && $C_{7,4}=(f,C_{3,2}^2)_3$ && && &\cr
\noalign{\hrule}
& 8  && $C_{8,2}=(C_{2,4},C_{3,2}^2)_3$ && && && &\cr
\noalign{\hrule}
& 9  && $C_{9,4}=(C_{3,8},C_{3,2}^2)_4$ && && && &\cr
\noalign{\hrule}
& 10  && $C_{10,0}=(f,C_{3,2}^3)_6$ && $C_{10,2}=(f,C_{3,2}^3)_5$ && && &\cr
\noalign{\hrule}
& 12  && $C_{12,2}=(C_{3,8},C_{3,2}^3)_6$ && && && &\cr
\noalign{\hrule}
& 15  && $C_{15,0}=(C_{3,8},C_{3,2}^4)_8$ && && && &\cr
} \hrule}
}}
\end{footnotesize}
 
\end{section}
\begin{section}{Covariants and modular forms}
The group $\Gamma_2$ acts on the Siegel upper half space $\mathfrak{H}_2$ and the orbifold
quotient $\Gamma_2 \backslash \mathfrak{H}_2$ can be identified with the moduli space
$\mathcal{A}_2$ of principally polarized abelian surfaces. If $\mathcal{M}_2$ denotes the
moduli space of complex smooth projective curves of genus $2$ we have the
 Torelli map $\mathcal{M}_2 \hookrightarrow \mathcal{A}_2$. 
This is an embedding and the complement of the image is the locus 
$\mathcal{A}_{1,1}$ of products of elliptic curves. 
This is the image of the `diagonal'
$$
\{ \tau=\left( \begin{matrix} \tau_{11} & \tau_{12} \\ \tau_{12} & \tau_{22}\\ \end{matrix}
\right) \in \mathfrak{H}_2 : \tau_{12}=0 \}
$$
and also the zero locus of the cusp form $\chi_{10}$ that vanishes with order $2$
there. 

The moduli space $\mathcal{M}_2$ has another description as a stack quotient 
of the action of ${\rm GL}(2,{\CC})$ on the space of binary sextics. 
We take the opportunity to correct an erroneous representation of this 
stack quotient in \cite{CFvdG}.  

Let $V$ be a $2$-dimensional vector space, say generated by $x_1,x_2$,
and consider ${\rm Sym}^6(V)$, the space of binary sextics. 
The group ${\rm GL}(V)$ acts from the right; an element $A=\left(\begin{matrix}
a & b \\ c & d \\ \end{matrix} \right)$ sends $f(x_1,x_2)$ to 
$f(ax_1+bx_2,cx_1+dx_2)$. We twist the action by $\det^{-2}(V)$ and consider
then 
$$
\mathcal{X}={\rm Sym}^6(V)\otimes {\det}^{-2}(V)\, .
$$ 
We let $\mathcal{X}^0\subset \mathcal{X}$ be the open set of binary sextics with non-vanishing
discriminant. An element $f$ of $\mathcal{X}^0$ defines a nonsingular
curve of genus $2$ via the equation $y^2=f(x)$. The action on the 
equation $y^2=f(x)$ is now induced by
$$
x \mapsto (ax+b)/(cx+d), \quad y \mapsto (ad-bc)\,  y/(cx+d)^3 \, .
$$
Then $\eta {\rm id}_V$ acts on the binary sextics as $\eta^2$, so that 
only $\pm {\rm id}_V$ acts trivially. 
The action of $-{\rm id}_V$ on  $(x,y)$ is $(x,y) \mapsto (x,-y)$ 
and induces the
hyperelliptic involution. 
So the stack quotient $[\mathcal{X}^0/{\rm GL}(V)]$
equals the stack $\mathcal{M}_2$. Let $\alpha:
\mathcal{X}^0 \to \mathcal{M}_2$ be the quotient map.

The equation $y^2=f(x)$ defines two differentials $xdx/y$ and $dx/y$ that
form a basis of the space of regular differentials on the curve and the 
action of ${\rm GL}(V)$ 
is by the standard representation. Thus the pullback
under $\alpha$
of the Hodge bundle ${\EE}$ 
from $\mathcal{M}_2$ to $\mathcal{X}^0$ is the equivariant
bundle defined by the standard representation $V\times \mathcal{X}^0$. 
The equivariant bundle ${\rm Sym}^6(V)\otimes
\det^{-2}(V)$ has the diagonal section $f \mapsto (f,f)$. 
This diagonal section, the universal binary sextic, thus defines a meromorphic
section $\chi_{6,-2}$ of ${\rm Sym}^6({\EE}) \otimes \det({\EE})^{-2}$. 
Since the construction extends to the locus of binary sextics with zeroes
of multiplicity at most $2$, the section extends regularly over
$\delta_0\setminus\delta_1$. 
(Here, $\delta_0$ corresponds to $\overline{\mathcal{A}}_2\setminus
\mathcal{A}_2$, the divisor at infinity, and $\delta_1$ to the closure
of $\mathcal{A}_{1,1}$.)
With this construction, the pole order
along $\delta_1$ is not yet known, but after multiplication with a power
of $\chi_{10}$ the section becomes regular.

In fact, it is not hard to see that $\chi_{6,-2}$ has a simple pole
along $\delta_1$. Using Taylor or series expansions in the normal
direction to $\mathfrak{H}_1\times\mathfrak{H}_1$ with coordinate
$t=2\pi i\tau_{12}$ as in \cite[\S5]{CFvdG} and coordinates $c_i$ on $\Sym^j$
corresponding to the monomials $\binom{j}{i}x_1^{j-i}x_2^i$,
we see that the coefficient of $t^m$ in $c_i$ 
in the expansion of a section of ${\rm Sym}^j({\EE})\otimes \det({\EE})^{\otimes k}$ 
is of the form $g\otimes h$,
with $g$ quasimodular of weight $j-i+k+m$ and $h$ quasimodular of weight $i+k+m$.
To get nonzero coefficients, the two weights and hence their sum $j+2k+2m$ must
be nonnegative. For $\chi_{6,-2}$, we get $2+2m\ge0$, hence $m\ge-1$, proving
the claim. Multiplying $\chi_{6,-2}$ with $\chi_{10}$, we obtain the
holomorphic modular form $\chi_{6,8}$, unique up to a scalar; alternatively,
$\chi_{6,-2}$ can be written as $\chi_{6,3}/\chi_5$.

We can interpret modular forms as
sections of vector bundles made out of ${\EE}$ by Schur functors, 
like ${\rm Sym}^j({\EE})\otimes
\det({\EE})^{\otimes k}$.  
Since the pullback of the Hodge bundle is the equivariant bundle defined by $V$,
the pullback of such a section can be interpreted as a
covariant.  
Recall that the ring of covariants is the
ring of invariants for the action of ${\rm SL}(V)$ on 
$V \oplus {\rm Sym}^6(V)$, see for example
\cite[p.~55]{Springer}.
Conversely, a (bihomogeneous) covariant corresponds to a meromorphic
modular form, with poles at most along $\delta_1$, hence to an element
of $M_{\chi_{10}}$.

We thus get maps
$$
M \to \mathcal{C} {\buildrel \nu \over \longrightarrow}  M_{\chi_{10}}
$$
with $\mathcal{C}$ the ring of covariants of binary sextics and
$M=\oplus_{j,k} M_{j,k}(\Gamma_2)$ and $M_{\chi_{10}}$ its localization 
at the multiplicative system generated by $\chi_{10}$. 
For another perspective on the map $\nu$, see \cite[\S6]{CFvdG}.

\end{section}
\begin{section}{The Order of Vanishing}
In this section we will describe a way to calculate the order of vanishing
along the locus $\mathcal{A}_{1,1}$ of a modular form defined by a covariant. 
A covariant $C$ has a bi-degree $(a,b)$:
if we consider $C$ as a form in the variables $a_0,\ldots,a_6$ and $x_1,x_2$
then it is of degree  $a$ in the $a_i$ and degree $b$ in $x_1,x_2$. The map
$\nu: \mathcal{C} \to M_{\chi_{10}}$ associates to $C$ a meromorphic modular form
of weight $(b, a-b/2)$ on $\Gamma_2$. It has the property that $\chi_5^{a} \, \nu(C)$
is a holomorphic modular form on $\Gamma_2$, but with character if $a$ is odd. 

Recall that $\mathcal{M}_2$ is represented as the stack
quotient $[\mathcal{X}^0/{\rm GL}(V)]$. The relation with 
the compactification of $\mathcal{M}_2$ is as follows.

In the (projectivized) space of binary sextics ${\PP}({\mathcal X})$ 
the discriminant defines
a hypersurface~$\Delta$. This hypersurface has a codimension $1$ singular locus,
one component of which
is the locus $\Delta'$ of binary sextics with three coinciding roots. 
So we are in codimension $2$ in ${\PP}(\mathcal{X})$ 
and we take a general plane $\Pi$ in 
${\PP}({\mathcal X})$ 
intersecting $\Delta$ transversally at a general point of $\Delta'$.   

In the plane $\Pi$ the intersection with $\Delta$ 
gives rise to a curve with a cusp singularity
corresponding to the intersection with $\Delta'$; 
we assume this latter point is the origin of $\Pi$.
In local coordinates $u,v$ in the  plane the discriminant is given by $u^2=v^3$. 
One then blows up the plane at the origin three times. This is illustrated in the
following picture (cf.\ the picture in \cite[p.\ 80]{D-S}).

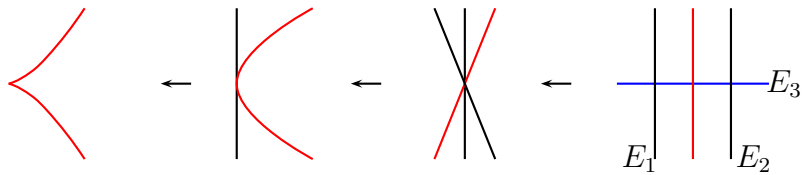
\begin{pspicture}(-2,-2)(10,3)
\psecurve[linecolor=red](-1,0)(0,0)(0.3,0.16)(0.5,0.35)(1,1)(2,2.82)
\psecurve[linecolor=red](-1,0)(0,0)(0.3,-0.16)(0.5,-0.35)(1,-1)(2,-2.82)
\psline{<-}(2,0)(2.4,0)
\psline{}(3,-1)(3,1)
\pscurve[linecolor=red](4,-1)(3,0)(4,1)
\psline{<-}(4.5,0)(4.9,0)
\psline(6,-1)(6,1)
\psline[linecolor=red](5.6,-1)(6.4,1)
\psline(6.4,-1)(5.6,1)
\psline{<-}(7,0)(7.4,0)
\psline[linecolor=blue](8,0)(10,0)
\psline[linecolor=red](9,-1)(9,1)
\psline(8.5,-1)(8.5,1)
\psline(9.5,-1)(9.5,1)
\rput(8.3,-1){$E_1$}
\rput(9.8,-1){$E_2$}
\rput(10.2,0){$E_3$}
\end{pspicture}

Then one blows down the exceptional fibres $E_1$ and $E_2$. The image of $E_3$ corresponds
in $\overline{\mathcal M}_2$ (resp.~$\overline{\mathcal A}_2$) to the locus 
$\delta_1$ (resp.~$\overline{\mathcal{A}}_{1,1}$) of unions (resp.~products) of elliptic curves. 

If $C$ is a covariant then it defines a section of an equivariant vector bundle 
on $\mathcal{X}$ and we can pull this back to the blow-up. It then makes sense to 
speak of the order of this section along the divisor $E_3$.

If we consider in the last setting a vertical line that intersects 
the image of $E_3$ transversally
at a general point, 
then this corresponds in the original plane with $u,v$ coordinates
to a curve $u^2=c\, v^3$. We can calculate the order of vanishing along $E_3$ by
calculating the order of the covariant on a general family corresponding to $u^2=c\, v^3$. 

The plane $\Pi$
corresponds to a family of binary sextics of the form
$$
g= (x^3+vx+u) h
$$
with $h$ a general cubic polynomial in $x$. The substitution $u=c^2t^3$, $v=ct^2$
(with $c$ general) gives a family corresponding to $u^2=c\, v^3$ 
and the order in $t$ of the covariant after substitution gives the order along $E_3$.

\begin{theorem} \label{thm1}
Let $C$ be a covariant of binary sextics of degree $a$ in the $a_i$ 
and let $\chi_C= \nu(C)$ be the
meromorphic modular form obtained by substituting $\chi_{6,-2}$. Then the order of $\chi_C$
along $\mathcal{A}_{1,1}$ is given by 
$$
{\rm ord}_{\mathcal{A}_{1,1}}(\chi_C)= 2 \, {\rm ord}_{E_3}(C) -a \, .
$$
\end{theorem}
\begin{proof}
Since $\chi_C$ is obtained by substituting the components of $\chi_{6,-2}$ in $C$
(cf.~\cite[\S6]{CFvdG}) and since $\chi_{6,-2}$ has a simple pole along $\delta_1$,
the order of $\chi_C$ along $\delta_1$ (a.k.a.~$\overline{\mathcal{A}}_{1,1}$) is at
least $-a$. It can only be larger when $C$ vanishes along $E_3$, the exceptional divisor
of the third blow-up of $\mathcal{X}$. To work this out precisely, note first that
the degree (resp.~the order) of a product equals the sum of the degrees (resp.~the orders)
of the factors. Hence, after replacing $C$ by its square if necessary, we may assume
that $a$ is even, equal to $2c$.
Consider the invariant $A$ of degree~$2$:
$$A=a_0a_6-6a_1a_5+15a_2a_4-10a_3^2$$
(proportional to $C_{2,0}$). Clearly, it doesn't vanish on $E_3$, and the associated
scalar-valued meromorphic modular form $\chi_A$ of weight~$2$ has a pole of order~$2$
along~$\delta_1$. We can write $C$ as $(C/A^c)\cdot A^c$ and $\chi_C$ as
$\chi_{C/A^c}\cdot \chi_A^c$, where $C/A^c$ is a meromorphic covariant and
$\chi_{C/A^c}$ a meromorphic vector-valued modular form, regular along $\delta_1$
but with possible poles along the zero locus of $\chi_A$.
The components of $C/A^c$ are meromorphic functions on
${\PP}({\mathcal X})$ that descend to the components of $\chi_{C/A^c}$. The
(minimal) orders of vanishing along $E_3$ respectively $\delta_1$ are clearly
closely related, but since $E_3$ in the picture above corresponds to the
{\em coarse} moduli space $M_{1,1}$, not to the stack ${\mathcal M}_{1,1}$,
the order of $\chi_{C/A^c}$ along $\delta_1$ equals twice the order of $C/A^c$ along $E_3$.
\end{proof}
\end{section}
\begin{section}{Rings and Modules of Modular Forms}

Let $R=\oplus_{k \, {\rm even}} M_k(\Gamma_2)$ be the graded 
ring of scalar-valued Siegel modular forms of even weight on $\Gamma_2$.
One knows that
$
R=\CC[E_4,E_6,\chi_{10},\chi_{12}]
$
and so its Hilbert-Poincar\' e series equals
$1/(1-t^4)(1-t^6)(1-t^{10})(1-t^{12})$.

We denote by $\epsilon$ the unique nontrivial character of order 2 
of $\Gamma_2$  (see Section 12 for a description of this character). 
Let $\Gamma_2[2]$ be the principal congruence subgroup of level $2$ of $\Gamma_2$.
The group ${\rm Sp}(4,{\ZZ}/2{\ZZ})$ is isomorphic to $\mathfrak{S}_6$.
We fix an explicit isomorphism
by identifying the symplectic lattice over ${\ZZ}/2{\ZZ}$ with the subspace
$\{ (a_1,\ldots,a_6) \in ({\ZZ}/2{\ZZ})^6: \sum a_i=0\}$ modulo the
diagonally embedded ${\ZZ}/2{\ZZ}$ with form $\sum_i a_ib_i$ as in
\cite[Section 2]{BFvdG1}; it is given explicitly on generators
of $\mathfrak{S}_6$ in \cite[Section 3, (3.2)]{CvdGG}.
Thus $\mathfrak{S}_6$  
acts on the space of modular forms $M_{j,k}(\Gamma_2[2])$
and the space $M_{j,k}(\Gamma_2,\epsilon)$ can be identified with  the subspace
of $M_{j,k}(\Gamma_2[2])$ on which $\mathfrak{S}_6$ acts via the alternating representation.
Since $-1_4$ belongs to $\Gamma_2[2]$, we have $M_{j,k}(\Gamma_2,\epsilon)=(0)$ for $j$ odd.
In the sequel, the integer $j$ will always be even.
The following result is in \cite{IW}; 
for the reader's convenience we give an alternative proof.
\begin{lemma}
We have $M_{j,k}(\Gamma_2,\epsilon)=S_{j,k}(\Gamma_2,\epsilon)$ for $(j,k)\neq (0,0)$.
\end{lemma}
\begin{proof}
In case $k=0$ and $j\neq 0$ it is well-known that $M_{j,0}(\Gamma_2,\epsilon)=(0)$,
see \cite[Satz1]{Freitag}. 
The Siegel operator $\Phi_2$ maps $M_{j,k}(\Gamma_2[2])$ to $S_{j+k}(\Gamma_1[2])$ which is
$(0)$ if $k$ is odd and $j$ is even. 
Since
$
M_{j,k}(\Gamma_2,\epsilon)
\subseteq
M_{j,k}(\Gamma_2[2])
$
we find $M_{j,k}(\Gamma_2,\epsilon)=S_{j,k}(\Gamma_2,\epsilon)$ for $k$ odd. 
For $k\geq 2$ even, the Eisenstein part $E_{j,k}(\Gamma_2[2])$ 
of $M_{j,k}(\Gamma_2[2])$, that is, the orthogonal complement of 
$S_{j,k}(\Gamma_2[2])$,  was described in \cite[Section 13]{CvdGG}
as an $\mathfrak{S}_6$-representation. From the description there
we see that the isotypical 
component $s[1^6]$ never occurs in $E_{j,k}(\Gamma_2[2])$;
the result follows since 
$S_{j,k}(\Gamma_2,\epsilon)=S_{j,k}(\Gamma_2[2])^{s[1^6]}$.
(Note that there is a misprint in the expression in
\cite[Prop.\ 13.1]{CvdGG}: $\Sym^k$ should
be read as $\Sym^{(j+k)/2}$.)
\end{proof}

The preceding lemma allows us to study cusp forms only.
The dimensions of the spaces $S_{j,k}(\Gamma_2,\epsilon)$ are known
by work of Tsushima as completed by Bergstr\"om (see \cite{BFvdG}).
The next table gives the Hilbert-Poincar\' e series of 
$\mathcal{M}^{\rm odd}_j(\Gamma_2,\epsilon)$ and 
$\mathcal{M}^{\rm ev}_j(\Gamma_2,\epsilon)$ as $R$-modules. 
We give only the numerators since in all cases we have
\[
\sum_{k\equiv_2 0  \, (\text{or} \, 1 )} 
\dim S_{j,k}(\Gamma_2,\epsilon)\, t^k=
\frac{N_j}{(1-t^4)(1-t^6)(1-t^{10})(1-t^{12})}\, ,
\] 
with $N_j$ a polynomial in $t$.

\begin{footnotesize}
\smallskip
\vbox{
\bigskip\centerline{\def\quad{\hskip 0.6em\relax}
\def\quod{\hskip 0.3em\relax }
\vbox{\offinterlineskip
\hrule
\halign{&\vrule#&\strut\quod\hfil#\quad\cr
height2pt  &\omit &&\omit &&\omit   &\cr
& $j$ && $k\, \bmod 2 $ && $N_j(t)$  &\cr 
\noalign{\hrule}
height2pt  &\omit &&\omit &&\omit   &\cr
& 0 && $1$ &&  $t^{5}$  &\cr 
&    && $0$ && $t^{30}$   &\cr 
\noalign{\hrule}
height2pt  &\omit &&\omit &&\omit   &\cr
& 2 && $1$  && $t^{9}+t^{11}+t^{17}$   &\cr 
&    && $0$ && $t^{16}+t^{22}+t^{24}$   &\cr 
\noalign{\hrule}
height2pt  &\omit &&\omit &&\omit   &\cr
& 4 && $1$ && $t^{9}+t^{11}+t^{13}+t^{15}+t^{17}$  &\cr 
&    && $0$ && $t^{14}+t^{16}+t^{18}+t^{20}+t^{22}$  &\cr 
\noalign{\hrule}
height2pt  &\omit &&\omit &&\omit   &\cr
& 6 && $1$ && $t^{3}+t^{5}+t^{11}+t^{13}+t^{17}+t^{19}+t^{21}$   &\cr 
&    && $0$ && $t^{8}+t^{10}+t^{12}+t^{16}+t^{18}+t^{24}+t^{26}$   &\cr 
\noalign{\hrule}
height2pt  &\omit &&\omit &&\omit   &\cr
& 8 && $1$ && $t^{5}+t^{7}+2\,t^{9}+t^{11}+t^{13}+t^{15}+t^{17}+t^{23}$  &\cr 
&    && $0$ && $t^{4}+t^{10}+t^{12}+t^{14}+t^{16}+2\, t^{18}+t^{20}+t^{22}$  &\cr 
\noalign{\hrule}
height2pt  &\omit &&\omit &&\omit   &\cr
& 10 && $1$ && $t^{5}+t^{7}+2\,t^{9}+2t^{11}+2t^{13}+2t^{15}+t^{17}$   &\cr 
&      && $0$ && $t^{8}+2\, t^{10}+2\,t^{12}+2t^{14}+2t^{16}+t^{18}+t^{20}$   &\cr 
\noalign{\hrule}
height2pt  &\omit &&\omit &&\omit   &\cr
& 12 && $1$ && $t^{3}+2\,t^{5}+t^{7}+2\, t^{9}+3\, t^{11}+2\, t^{13}+t^{17}+t^{19}-t^{23}+t^{27}$   &\cr 
&      && $0$ && $t^{2}+t^{4}+t^{6}+t^{8}+t^{10}+t^{12}+t^{14}+2\, t^{16}+2\, t^{18}+t^{20}+t^{22}+t^{24}-t^{28}$   &\cr 
} \hrule}
}}
\end{footnotesize}

\noindent 
For $j\in \left\{0,2,4,6,8,10\right\}$ and both for $k$ odd and even
the shape of the polynomials $N_j$ is 
as follows:
\[
N_{j}(t)=a_{k_{j,1}}\, t^{k_{j,1}}+\ldots+a_{k_{j,n}}\, t^{k_{j,n}}
\quad
\text{with}
\quad
n, a_{k_{j,i}} \in \ZZ_{>0}
\quad
\text{and}
\quad
\sum_{i=1}^{n}a_{k_{j,i}}=j+1.
\]
This suggests that the $R$-modules $\mathcal{M}_j^{\rm ev}(\Gamma,\epsilon)$ 
and  $\mathcal{M}_j^{\rm odd}(\Gamma,\epsilon)$ 
are generated by $j+1$ cusp forms 
with $a_{j,k_{j,i}}$  
generators of weight $(j,k_{j,i})$. As the table shows this does not hold
for $j=12$.

Therefore the strategy of the proof for the structure of the modules
will be to show first that 
there is no cusp form of weight $(j,k)$ for $k<k_{j,1}$ for 
$j\in \left\{0,2,4,6,8,10\right\}$. In the cases at hand 
this follows from the above
formula and the results in \cite{C-vdG-Ch}.
Then we will construct $j+1$ cusp forms and check 
that their wedge product is not identically $0$. 
In fact in all cases we find that the wedge product of the $j+1$ forms
is a nonzero multiple of a product of powers of $\chi_5$ and $\chi_{30}$.
This proves that the submodule they generate 
has the same Hilbert-Poincar\' e series as 
the whole module, hence that we found the whole module.
We will give the covariants that define the generators explicitly 
in a number of cases, but in view of their size we refer for the other cases
to \cite{BFvdG} where we will make these available.
\end{section}
\begin{section}{The scalar-valued cases}
In this section we deal with the modules of scalar-valued modular forms
with character. In this case the weight $(j,k)$ is of the form $(0,k)$
and we simply indicate it  by~$k$.

The diagonal element $\gamma_1={\rm diag}(1,-1,1,-1) \in \Gamma_2$ defines an
involution fixing the coordinates $\tau_{11}$ and $\tau_{22}$ and replacing $\tau_{12}$ by
$-\tau_{12}$. Its fixed point set is 
the locus defined by $\tau_{12}=0$. This defines the Humbert surface $H_1=\mathcal{A}_{1,1}$
parametrizing products of elliptic curves in~$\mathcal{A}_2$. 
There is another involution $\iota_2$ given by
$\gamma_2=(a,b;c,d)$ with $b=c=0$ and 
$a=d=\left( \begin{smallmatrix} 0 & 1 \\ 1 & 0 \\ \end{smallmatrix} \right)$ 
which interchanges $\tau_{11}$ and
$\tau_{22}$, but fixes $\tau_{12}$. The fixed point set of $\iota_2$
is the locus $\tau_{11}=\tau_{22}$ and defines the Humbert surface 
$H_4$ in~$\mathcal{A}_2$, see \cite{vdG}. 
One checks that the action on modular forms is as follows
$$
\gamma_1: f \mapsto (-1)^k \, f, \quad
\gamma_2: f \mapsto  (-1)^{k+1} \, f \qquad \text{for $f \in M_k(\Gamma_2,\epsilon)$.}
\eqno(1)
$$
Note $\epsilon(\gamma_2)=-1$.
It follows that $f\in M_k(\Gamma_2,\epsilon)$ vanishes on $H_1$ for $k$ odd
and on $H_4$ for $k$ even.

We have two modular forms $\chi_5$ 
and $\chi_{30}$ of weight $5$ and $30$ whose zero loci in $\mathcal{A}_2$ equal
$H_1$ and $H_4$. We recall their construction.

The cusp form $\chi_5\in S_{5}(\Gamma_2,\epsilon)$ 
is defined in terms of theta functions. 
For $(\tau,z) \in \mathfrak{H}\times \CC$ and 
$(\mu_1,\mu_2)$, $(\nu_1,\nu_2)$ in $\ZZ^2$ we have the standard
theta series with characteristics
$$
\vartheta_
{
\left[
\begin{smallmatrix}
\mu\\
\nu
\end{smallmatrix}
\right]
}
(\tau,z)=
\sum_{n=(n_1,n_2)\in \ZZ^2}
e^
{
i \pi 
(n+\mu/2)
(\tau\, (n+\mu/2)^t+2(z+\nu/2))
}.
$$
By letting $\mu$ and $\nu$ be vectors consisting of zeroes and ones with
$\mu^t \nu \equiv 0 \, (\bmod \, 2)$ and setting $z=0$ we obtain ten so-called
theta constants
and their product defines a cusp form of weight $5$ on $\Gamma_2$ with
character $\epsilon$:
$$
\chi_5=-\frac{1}{64}
\, \prod
\vartheta_
{
\left[
\begin{smallmatrix}
\mu\\
\nu
\end{smallmatrix}
\right]
} \, .
$$
Its Fourier expansion starts with
$$
\chi_5(\tau)=(u-1/u)XY+\ldots
$$
where $X=e^{\pi i \tau_1}$, $Y=e^{\pi i \tau_2}$ and $u=e^{\pi i \tau_{12}}$.
We note that $\chi_5^2=\chi_{10}$ and the vanishing locus of $\chi_{10}$ 
in $\mathcal{A}_2$ is $2H_1$.

In order to construct $\chi_{30}$ we consider the invariant $C_{15,0}$,
given in the table in Section 2. 
By the procedure of \cite{CFvdG} it provides a meromorphic cusp form 
of weight $15$ on $\Gamma_2$. 
One checks using Theorem~\ref{thm1} that the order of this form 
along $\mathcal{A}_{1,1}$
is $-3$. So we obtain a holomorphic modular form by multiplying by
$\chi_5^3$ and we set
$$
\chi_{30}= 2^{-11} 3^{11}\cdot 5^{11} \cdot 11 \cdot 13 \,   \nu(C_{15,0}) 
\chi_5^{3}\,  ;
$$
it is a cusp form in $S_{30}(\Gamma_2,\epsilon)$ whose
Fourier expansion starts with
$$
\chi_{30}(\tau)=(u+1/u)X^3Y^5-(u+1/u)X^5Y^3+\ldots\, .
$$
\begin{theorem}\label{Graded0}
We have
$
\mathcal{M}_0^{\rm odd}(\Gamma_2,\epsilon)= R\, \chi_5$ and
$\mathcal{M}_0^{\rm ev}(\Gamma_2,\epsilon)= R\, \chi_{30}$.
\end{theorem}
\begin{proof}
Clearly $\mathcal{M}_0^{\rm odd}(\Gamma_2,\epsilon)$ contains $R\, \chi_5$ and
$\mathcal{M}_0^{\rm ev}(\Gamma_2,\epsilon)$ contains $R\, \chi_{30}$.
The generating function for the dimensions shows that $\chi_5$ 
(resp.\ $\chi_{30}$) generates.
\end{proof}

\begin{remark}
We know the cycle classes of the closures of $H_1$ and $H_4$ in
the compactified moduli space $\tilde{\mathcal{A}}_2$. In the divisor class
group with rational coefficients of $\tilde{\mathcal{A}}_2$ we have
$$
5\lambda_1 = [\overline{H}_1]+[D], \quad 30\lambda_1=[\overline{H}_4]+[D]
$$
with $D$ the divisor at infinity of $\tilde{\mathcal{A}}_2$, and $\lambda_1$ the
first Chern class of the determinant of the Hodge bundle, 
see \cite[Thm.~2.6]{vdG}.
From this it follows that the vanishing locus of $\chi_{30}$ in $\mathcal{A}_2$
is $H_4$. Then (1) implies that for $k$ odd (resp.\ $k$ even) 
any $f\in M_k(\Gamma_2,\epsilon)$
is divisible by $\chi_5$ (resp.\ by $\chi_{30}$). 
This implies the theorem as well.
\end{remark}

For later identifications (for example in the proof of Theorem \ref{thmj=6})
we need the restriction of $\chi_{6,3}$ to the Humbert
surface $H_4$. This surface can be given by $\tau_{11}=\tau_{22}$, or
equivalently by $\tau_{12}=1/2$. Let $\chi$ denote the Dirichlet character 
modulo $4$ defined
by the Kronecker symbol $\left( \frac{-4}{\cdot} \right)$. 
The space $S_3^{\rm new}(\Gamma_0(16),\chi)$ is generated by $\eta^6(2\tau)$.
The 
space $S_5^{\rm new}(\Gamma_0(16),\chi)$ has dimension $2$ and a basis of eigenforms 
$g', g^{\prime\prime}$ with Fourier expansions 
$$
q - 8\sqrt{-3} \, q^3 +18\, q^5 -16\sqrt{-3}\, q^7 - 111\, q^9+\ldots 
$$
and similarly $S_{7}^{\rm new}(\Gamma_0(16),\chi)$ has dimension $2$
and a basis of eigenforms $f^{\prime}, f^{\prime\prime}$ with Fourier
expansions 
$$
q - 16\sqrt{-3}\, q^3 - 150\, q^5 - 352\sqrt{-3} \, q^7 - 39\, q^9+\ldots 
$$

\begin{lemma} The restriction of $\chi_{6,3}$ to $H_4$ is given by
$$
\chi_{6,3}\left(
\begin{smallmatrix}
\tau_1 & 1/2\\
1/2 & \tau_2
\end{smallmatrix}
\right)=2\, i\, 
\left[
\begin{smallmatrix}
16\, \eta^{18}(2\, \tau_1) \otimes \eta^6(2\,\tau_2)\\
0\\
F_1(\tau_1) \otimes F_2(\tau_2)\\
0\\
F_2(\tau_1) \otimes F_1(\tau_2)\\
0\\
16\, \eta^{6}(2\, \tau_1) \otimes \eta^{18}(2\, \tau_2)
\end{smallmatrix}
\right]
$$
where
$$
F_1=\frac{3+\sqrt{-3}}{6}\, f'+\frac{3-\sqrt{-3}}{6}\, f^{\prime\prime}
\quad
\text{and}
\quad
F_2=\frac{3+\sqrt{-3}}{6}\, g'+\frac{3-\sqrt{-3}}{6}\, g^{\prime\prime} \, .
$$
\end{lemma} 

\end{section}
\begin{section}{The case $j=2$}

We start with the case $k$ odd.
\begin{theorem}
The $R$-module $\mathcal{M}_2^{\rm odd}(\Gamma_2,\epsilon)$ is  free 
with three generators of weight $(2,9)$, $(2,11)$ and $(2,17)$.
\end{theorem}
\begin{proof}
We recall that the numerator $N_2$ of the Hilbert-Poincar\'e series is $t^9+t^{11}+t^{17}$.
We construct the three generators by considering the covariants
$$
\begin{aligned}
\xi_1&=4\, C_{2,0}C_{3,2}-15\, C_{5,2},  \\
\xi_2&=32 \, C_{2,0}^2 C_{3,2}+135 \, C_{2,0}C_{5,2}-300\, C_{3,2}C_{4,0} - 15750 \, C_{7,2},  \\
\xi_3&=C_{3,2}.  \\
\end{aligned}
$$
These three covariants define meromorphic modular forms
vanishing with order $-1$, $-1$, $-3$ along $\mathcal{A}_{1,1}$ 
(by Theorem \ref{thm1}), so we obtain holomorphic modular forms
$$
F_{2,9}= -\frac{3375}{4} \nu(\xi_1)\chi_5, \quad 
F_{2,11}=-\frac{10125}{8} \nu(\xi_2)\chi_5,  \quad 
F_{2,17}=\frac{1125}{2}\nu(\xi_3)\chi_5^3
$$
of weights $(2,9)$, $(2,11)$ and $(2,17)$ and their
Fourier expansions start as
$$
\begin{footnotesize}
F_{2,9}=
\left(
\begin{smallmatrix}
u - 1/u\\
u + 1/u\\
u - 1/u
\end{smallmatrix}
\right)XY+\ldots
\qquad
F_{2,11}=
\left(
\begin{smallmatrix}
u - 1/u\\
u + 1/u\\
u - 1/u
\end{smallmatrix}
\right)XY+\ldots
\end{footnotesize}
$$
and
$$
F_{2,17}=\begin{footnotesize}
\left(
\begin{smallmatrix}
{u}^{3}+9\,u-9\,{u}^{-1}-{u}^{-3}\\
{u}^{3}+71\,u +71\,{u}^{-1}+{u}^{-3}\\
{u}^{3}+9\,u-9\,{u}^{-1}-{u}^{-3} \\
\end{smallmatrix}
\right)X^3Y^3
+\ldots
\end{footnotesize}
$$
To prove the theorem we have to show that these three generators satisfy
$$
F_{2,9}\wedge F_{2,11}\wedge F_{2,17}\neq 0.
$$
Note that $\det({\rm Sym}^j({\EE}))=\det({\EE})^{j(j+1)/2}$, so this is a form in
$S_{40}(\Gamma_2,\epsilon)$.
The Fourier expansion of $F_{2,9} \wedge F_{2,11} \wedge F_{2,17}$ starts with
\[
86400 \, 
(
( -{u}^{3}+u+{u}^{-1}-{u}^{-3}) {Y}^{7}{X}^{5}+( {u}^{3}-u-{u}^{-1}+{u}^{-3}) {Y}^{5}{X}^{7}+ 
\ldots
)
\]
and this shows the result.
\end{proof}

\begin{remark}
The space $S_{40}(\Gamma_2,\epsilon)$
is $2$-dimensional, generated by $\chi_{5}^2\chi_{30}$ and $E_4E_6\chi_{30}$.
We check that 
$ F_{2,9} \wedge F_{2,11} \wedge F_{2,17}=-86400 \, \chi_5^2 \chi_{30}$.
\end{remark}

The case $k$ even is similar.

\begin{theorem}
The $R$-module $\mathcal{M}^{\rm ev}_2(\Gamma_2,\epsilon)$ is free
with generators of weight $(2,16)$, $(2,22)$ and $(2,24)$.
\end{theorem}

\begin{proof}
We use the covariants
$$
\begin{aligned}
\xi_1&=
1211\,C_{2,0}^{2}C_{8,2}-8910\,C_{2,0}C_{10,2}-5250\,C_{4,0}C_{8,2}+277200\,C_{12,2}, \\
\xi_2&=C_{8,2}, \quad
\xi_3=7\,C_{2,0}C_{8,2}-110\,C_{10,2} \\
\end{aligned}
$$
and set
\begin{footnotesize}
\begin{align*}
F_{2,16}=\frac{34171875}{2048} \nu(\xi_1) \chi_5&=
\left(
\begin{smallmatrix}
0\\
2(u - 1/u)\\
u + 1/u
\end{smallmatrix}
\right)XY^3+
\left(
\begin{smallmatrix}
-(u + 1/u)\\
-2(u - 1/u)\\
0
\end{smallmatrix}
\right)X^3Y+
\ldots\\
F_{2,22}=\frac{26578125}{8} \nu(\xi_2) \chi_5^3&=
\left(
\begin{smallmatrix}
u + 1/u\\
0\\
-(u +1/u)
\end{smallmatrix}
\right)X^3Y^3+\ldots
\\
F_{2,24}=-\frac{102515625}{16} \nu(\xi_3)\chi_5^3&=
\left(
\begin{smallmatrix}
u + 1/u\\
0\\
-(u +1/u)
\end{smallmatrix}
\right)X^3Y^3+\ldots
\end{align*}
\end{footnotesize}
By the criterion these are holomorphic modular forms of weight $(2,16)$,
$(2,22)$ and $(2,24)$.
The Fourier expansion of $F_{2,16} \wedge F_{2,22} \wedge F_{2,24}$ starts with
\[
F_{2,16} \wedge F_{2,22} \wedge F_{2,24}=-2880 \, 
(u^{3}+u-u^{-1}-u^{-3})\,  X^7Y^{11}+\ldots
\]
and in fact equals $-2880 \, \chi_5 \, \chi_{30}^2$.
This finishes the proof in view of the Hilbert-Poincar\'e series.
\end{proof}
\end{section}
\begin{section}{The case $j=4$.}

\begin{theorem}
The $R$-module $\mathcal{M}_4^{\rm odd}(\Gamma_2,\epsilon)$ is free with generators
of weight $(4,9)$, $(4,11)$, $(4,13)$, $(4,15)$ and $(4,17)$.
\end{theorem}
\begin{proof}
We use the covariants
$$
\begin{aligned}
\xi_1&=49\,C_{2,0}^{2}C_{2,4}+45\,C_{2,0}C_{4,4}-375\,C_{2,4}C_{4,0}-225\,C_{3,2}^{2},\\
\xi_2&=772\,C_{2,0}^{3}C_{2,4}-1260\,C_{2,0}^{2}C_{4,4}-4875\,C_{2,0}C_{2,4}C_{4,0}
-900\,C_{2,0}C_{3,2}^{2},\\
& \qquad -5625\,C_{2,4}C_{6,0}+13500\,C_{3,2}C_{5,2}+6750\,C_{4,0}C_{4,4}\\
\xi_3&=64\,C_{2,0}^{4}C_{2,4}-1200\,C_{2,0}^{2}C_{2,4}C_{4,0}-3600\,C_{2,0}^{2}C_{3,2}^{2}
+27000\,C_{2,0
}C_{3,2}C_{5,2}\\
& \qquad +5625\,C_{2,4}C_{4,0}^{2}-50625\,C_{5,2}^{2},\\
\xi_4&=C_{2,4}, \qquad \xi_5= 3\,C_{2,0}C_{2,4}-5\,C_{4,4}\, .\\
\end{aligned}
$$
The Fourier expansions of 
$$
F_{4,9}=-\frac{675}{4} \nu(\xi_1) \chi_5, \qquad 
F_{4,11}=\frac{2025}{8} \nu(\xi_2)\chi_5 \quad \text{and}\quad 
F_{4,13}=-\frac{30375}{8} \nu(\xi_3)\chi_5
$$
all three start as
$$
\left(
\begin{smallmatrix}
u-1/u\\
2(u + 1/u)\\
3(u - 1/u)\\
2(u + 1/u)\\
u-1/u
\end{smallmatrix}
\right)XY+
\ldots
$$
The other two modular forms we need are 
\begin{footnotesize}
\begin{align*}
F_{4,15}=\frac{75}{2} \nu(\xi_4) \chi_5^3&=
\left(
\begin{smallmatrix}
u^3-3u+3/u-1/u^3\\
2(u^3-u-1/u+1/u^3)\\
3(u^3+5u-5/u-1/u^3)\\
2(u^3-u-1/u+1/u^3)\\
u^3-3u+3/u-1/u^3
\end{smallmatrix}
\right)\, X^3\, Y^3+
\ldots\\
F_{4,17}=-\frac{675}{2} \nu(\xi_5) \chi_5^3&=
\left(
\begin{smallmatrix}
u^3+9u-9/u-1/u^3\\
2(u^3-u-1/u+1/u^3)\\
3(u^3-3u+3/u-1/u^3)\\
2(u^3-u-1/u+1/u^3)\\
u^3+9u-9/u-1/u^3
\end{smallmatrix}
\right)\, X^3\, Y^3+
\ldots
\end{align*}
\end{footnotesize}
The Fourier expansion of 
$F_{4,9} \wedge F_{4,11} \wedge F_{4,13} \wedge F_{4,15} \wedge F_{4,17}$ starts with
\[
-2866544640\, 
(u^{5}-u^{3}-2u+2/u+1/u^3-1/u^5) X^9Y^{13}+\ldots
\]
and by a calculation we get 
\[
F_{4,9} \wedge F_{4,11} \wedge F_{4,13} \wedge F_{4,15} \wedge F_{4,17}
=-2866544640\, \chi_{5}^3 \chi_{30}^2 \, .
\]
\end{proof}

\begin{theorem}
The $R$-module $\mathcal{M}_4^{\rm ev}(\Gamma_2,\epsilon)$ is free with generators
of weight $(4,14)$, $(4,16)$, $(4,18)$, $(4,20)$ and $(4,22)$.
\end{theorem}
\begin{proof}
For weight $(4,14)$ we consider the covariant $\xi_1$ given as
$$
189\,C_{2,0}^{3}C_{5,4}+12390\,C_{2,0}^{2}C_{7,4}
-750\,C_{2,0}C_{4,0}C_{5,4}-63000\,(C_{2,0}C
_{9,4}+C_{3,2}C_{8,2}+C_{4,0}C_{7,4})
$$
and set 
$F_{4,14}=-(151875/1024) \nu(\xi_1)\chi_5$.
This is holomorphic and its Fourier expansion starts with
\[
F_{4,14}(\tau)=
\left(
\begin{smallmatrix}
0\\
0\\
0\\
2(u-1/u)\\
(u+1/u)
\end{smallmatrix}
\right)XY^3-
\left(
\begin{smallmatrix}
(u+1/u)\\
2(u-1/u)\\
0\\
0\\
0
\end{smallmatrix}
\right)X^3Y+
\ldots
\]
For weight $(4,16)$ we consider the covariant $\xi_2$ given as
$$
\begin{aligned}
&
11176\,C_{2,0}^{4}C_{5,4}-82320\,C_{2,0}^{3}C_{7,4}+9576000\,C_{2,0}^{2}C_{9,4}
-15750\,C_{2,0}C_{3,2}C_{8,2}\\ 
& -220500\,C_{2,0}C_{4,0}C_{7,4}
-176625\,C_{2,0}C_{5,4}C_{6,0}-414000\,C_{4,0}^{2}C_{5,4}+43213500\,C_{3,2}C_{10,2} \\
&
-47250000\,C_{4,0}C_{9,4}
+20506500\,C_{5,2}C_{8,2}-9308250\,C_{6,0}C_{7,4} \\
\end{aligned}
$$
and set $F_{4,16}=(151875/4096) \nu(\xi_2)\chi_5$; it is holomorphic and its 
Fourier expansion starts with
\[
F_{4,16}(\tau)=
\left(
\begin{smallmatrix}
0\\
2(u+1/u)\\
3(u+1/u)\\
(u-1/u)\\
0
\end{smallmatrix}
\right)XY^3+
\ldots
\]
We get a form $F_{4,18}$ of weight $(4,18)$ by putting 
$F_{4,18}=(16875/8) \nu(C_{5,4})\chi_5^3$; it is holomorphic and its Fourier expansion starts with 
\[
F_{4,18}(\tau)=
\left(
\begin{smallmatrix}
3(u+1/u)\\
2(u-1/u)\\
0\\
-2(u-1/u)\\
-3(u+1/u)
\end{smallmatrix}
\right)X^3Y^3+
\ldots
\]
For weight $(4,20)$ we consider the covariant $\xi_4= C_{2,0}C_{5,4}+70 C_{7,4}$ and 
put $F_{4,20}=(151875/32) \nu(\xi_4)\chi_5^3$ with Fourier expansion
\[
F_{4,20}(\tau)=
\left(
\begin{smallmatrix}
0\\
(u-1/u)\\
0\\
-(u-1/u)\\
0
\end{smallmatrix}
\right)X^3Y^3+
\ldots
\]
Finally, the covariant $\xi_5=C_{2,0}^{2}C_{5,4}-10\,C_{2,0}C_{7,4}+1000\,C_{9,4}$
yields the form $F_{4,22}=(3189375/32)\nu(\xi_5) \chi_5^3$ with Fourier expansion
\[
F_{4,22}(\tau)=
\left(
\begin{smallmatrix}
(u+1/u)\\
2(u-1/u)\\
0\\
-2(u-1/u)\\
-(u+1/u)
\end{smallmatrix}
\right)X^3Y^3+
\ldots
\]
The Fourier expansion of 
$F_{4,14} \wedge F_{4,16} \wedge F_{4,18} \wedge F_{4,20} \wedge F_{4,22}$ 
starts with
\[
-20736\, 
(u^{5}+u^{3}-2u-2/u+1/u^3+1/u^5) X^{11}Y^{17}+\ldots
\]
and in fact we checked that it equals $-20736\, \chi_5^2 \chi_{30}^3$.
\end{proof}
\end{section}
\begin{section}{The case $j=6$}
\begin{theorem}\label{thmj=6}
The $R$-module $\mathcal{M}_6^{\rm odd}(\Gamma_2,\epsilon)$ 
is free with generators of weight
$(6,3)$, $(6,5)$, $(6,11)$, $(6,13)$, $(6,17)$, $(6,19)$ and $(6,21)$.
\end{theorem}
\begin{proof}
We use the covariants
$$
\begin{aligned}
\xi_1=&C_{1,6},\qquad
\xi_2=8\,C_{1,6}C_{2,0}-75\,C_{3,6},\\
\xi_3=&125\,C_{1,6}C_{2,0}^{2}C_{4,0}+249\,C_{1,6}C_{2,0}C_{6,0}-840\,C_{1,6}C_{4,0}^{2}-189\,C_{2,0}
C_{2,4}C_{5,2}\\
&-1008\,C_{2,0}C_{3,2}C_{4,4}-72\,C_{2,0}C_{3,6}C_{4,0}+630\,C_{3,2}^{3}+132300\,C_{2
,4}C_{7,2}\\
&+2430\,C_{3,6}C_{6,0}-1890\,C_{4,4}C_{5,2},\\
\xi_4=&768\,C_{1,6}C_{2,0}^{5}+768\,C_{2,0}^{4}C_{3,6}-487520\,C_{1,6}C_{2,0}^{2}C_{6,0}-36075\,C_{2,0}^{2}
C_{2,4}C_{5,2}\\
&+33600\,C_{2,0}^{2}C_{3,2}C_{4,4}-52500\,C_{2,0}C_{3,2}^{3}-11061300\,C_{1,6}C_{4,0}
C_{6,0}\\
&-314861750\,C_{2,0}C_{2,4}C_{7,2}-112500\,C_{2,0}C_{3,6}C_{6,0}+8956675\,C_{2,0}C_{4,4}C_{5,2
}\\
&+17767100\,C_{2,4}C_{3,2}C_{6,0}+230625\,C_{2,4}C_{4,0}C_{5,2}-39779100\,C_{3,2}^{2}C_{5,2}\\
&+17834600\,C_{3,2}C_{4,0}C_{4,4}+9482503800\,C_{1,6}C_{10,0}-932772750\,C_{4,4}C_{7,2},\\
\xi_5=&8\,C_{1,6}C_{2,0}^{2}-125\,C_{2,4}C_{3,2},\\
\xi_6=&128\,C_{1,6}C_{2,0}^{3}+6600\,C_{2,0}^{2}C_{3,6}+6750\,C_{2,4}C_{5,2}-9000\,C_{3,2}C_{4,4}-52875\,{
\it Cov}_{3,6}C_{4,0},\\
\xi_7=&
-837\,C_{1,6}C_{2,0}^{2}C_{4,0}+415\,C_{1,6}C_{2,0}C_{6,0}+9450\,C_{2,0}C_{2,4}C_{5,2}
+6075\,C_{2,0}C_{3,6}C_{4,0}\\
&+3150\,C_{3,2}^{3}
-1543500\,C_{2,4}C_{7,2}-17475\,C_{3,6}C_{6,0}+14175\,C_{4,4}C_{5,2}\,.
\end{aligned}
$$
We consider the following cusp forms:
$$
F_{6,3}=\nu(\xi_1)\chi_5,
\,\,
F_{6,5}=-15\, \nu(\xi_2)\chi_5,
\,\,
F_{6,11}=\frac{253125}{8}\nu(\xi_3)\chi_5,
\,\,
F_{6,13}=\frac{2278125}{16}\nu(\xi_4)\chi_5,
$$
and
$$
F_{6,17}=-\frac{675}{4}\nu(\xi_5)\chi_5^3,
\quad
F_{6,19}=-\frac{675}{2}\nu(\xi_6)\chi_5^3,
\quad
F_{6,21}=-\frac{151875}{4}\nu(\xi_7)\chi_5^3\,.
$$
Then  
$$
W_{110}=F_{6,3} \wedge F_{6,5} \wedge F_{6,11} \wedge F_{6,13}
\wedge F_{6,17} \wedge F_{6,19} \wedge F_{6,21}
$$ is a cusp form in $S_{0,110}(\Gamma_2,\epsilon)$ 
and its Fourier expansion starts with
$$
2^{30}\cdot 3^5\cdot 5^8\cdot 7^3\, 
(u^{7}-u^{5}-3u^3+3u+3/u-3/u^3-1/u^5+1/u^7) \, X^{13}Y^{17}+\ldots
$$
The order of vanishing of 
$W_{110}$ along $H_1$ is $4$ while along
$H_4$ it is $3$, so $W_{110}$ is a multiple of $\chi_5^4 \chi_{30}^3$ 
and a calculation at the level of covariants yields 
$W_{110}=2^{30}\cdot 3^5\cdot 5^8\cdot 7^3\,\chi_5^4\chi_{30}^3$~.
\end{proof}

\begin{theorem}
The $R$-module $\mathcal{M}_6^{\rm ev}(\Gamma_2,\epsilon)$ is free with generators of weight
$(6,8)$, $(6,10)$, $(6,12)$, $(6,16)$, $(6,18)$, $(6,24)$ and $(6,26)$.
\end{theorem}

\begin{proof}
We use the covariants
$$
\begin{aligned}
\xi_1&=16\,C_{2,0}C_{4,6}+75\,C_{6,6}^{(1)}-60\,C_{6,6}^{(2)},\quad
\xi_4=C_{4,6},\quad
\xi_5=4\, C_{2,0}C_{4,6}-15\,C_{6,6}^{(1)},\\
\xi_2&=-128\,C_{2,0}^{2}C_{4,6}+75\,C_{2,0}C_{6,6}^{(1)}-540\,C_{2,0}C_{6,6}^{(2)}-1500\,C_{3,2}C_{5,4}+1800\,C_{4
,0}C_{4,6},\\
\xi_3&=64\,C_{2,0}^{3}C_{4,6}-3975\,C_{2,0}^{2}C_{6,6}^{(1)}+1740\,C_{2,0}^{2}C_{6,6}^{(2)}-189000\,C_{2,4}C_{8,2}+
63000\,C_{3,2}C_{7,4}\\
&+40500\,C_{4,0}C_{6,6}^{(1)}-18000\,C_{4,0}C_{6,6}^{(2)}+4500\,C_{5,2}C_{5,4},\\
\xi_6&=-17472\,C_{2,0}C_{2,4}C_{8,2}+31360\,C_{2,0}C_{3,2}C_{7,4}-513\,C_{2,0}C_{4,0}C_{6,6}^{(1)}
+180\,C_{2,0}C_{4,0}C_{6,6}^{(2)}\\
&-64\,C_{2,0}C_{4,6}C_{6,0}+342\,C_{2,0}C_{5,2}C_{5,4}+39600\,C_{2,
4}C_{10,2}-126000\,C_{3,2}C_{9,4}\\
&-16800\,C_{4,4}C_{8,2}
-60900\,C_{5,2}C_{7,4}+600\,C_{6,0}C^{(1)}_{6,
6},\\
\xi_7&=
1024\,C_{2,0}^{5}C_{4,6}
-257152000\,C_{2,0}^{2}C_{3,2}C_{7,4}
+5375048250\,C_{2,0}C_{2,4}C_{10,2}\\
&-1808283750\,C_{2,0}C_{3,2}C_{9,4}
+785335250\,C_{2,0}C_{4,4}C_{8,2}
+1144763375\,C_{2,0}C_{5,2}C_{7,4}\\
&+673186500\,C_{2,4}C_{4,0}C_{8,2}
+656687500\,C_{3,2}^{2}C_{8,2}
-938905625\,C_{3,2}C_{4,0}C_{7,4}\\
&+3150000\,C_{4,0}^{2}C_{6,6}^{(2)}
+17435250\,C_{4,0}C_{5,2}C_{5,4}
-378064302000\,C_{2,4}C_{12,2}\\
&-532125000\,C_{4,4}C_{10,2}
-415800000\,C_{4,6}C_{10,0}
+37292797500\,C_{5,2}C_{9,4}\\
&-250254270000\,C_{7,2}C_{7,4}\,.
\end{aligned}
$$
We consider the following cusp forms:
$$
F_{6,8}=\frac{10125}{8}\nu(\xi_1)\chi_5,
\quad
F_{6,10}=-\frac{30375}{16}\nu(\xi_2)\chi_5,
\quad
F_{6,12}=\frac{455625}{64}\nu(\xi_3)\chi_5,
$$
$$
F_{6,16}=-3375\, \nu(\xi_4)\chi_5^3,
\quad
F_{6,18}=-50625\,\nu(\xi_5)\chi_5^3
\quad
F_{6,24}=-\frac{170859375}{32}\nu(\xi_6)\chi_5^3,
$$
$$
F_{6,26}=-\frac{20503125}{16}\nu(\xi_7)\chi_5^3\,.
$$
Then
\[
W_{135}=F_{6,8} \wedge F_{6,10} \wedge F_{6,12} \wedge F_{6,16} \wedge F_{6,18} \wedge F_{6,24} \wedge F_{6,26} 
\]
is a cusp form in $S_{135}(\Gamma_2,\epsilon)$ and its Fourier expansion starts with
\[
-2^{32}\cdot 3^{8}\cdot 5^8\cdot 7^2 \cdot 13 \cdot 23\, 
(u^{7}+u^{5}-3u^3-3u+3/u+3/u^3-1/u^5-1/u^7) X^{15}Y^{23}+\ldots
\]
A calculation shows that the order of vanishing of  $W_{135}$ along $H_1$ is $3$, 
while along $H_4$ it is $4$, so $W_{135}$ is a multiple of $\chi_5^3\chi_{30}^4$
and a calculation at the level of covariants tells us
\[
W_{135}=-2^{32}\cdot 3^{8}\cdot 5^8\cdot 7^2 \cdot 13 \cdot 23\, \chi_5^3\chi_{30}^4\, .
\]
\end{proof}

\end{section}
\begin{section}{The case $j=8$}

\begin{theorem}
The $R$-module $\mathcal{M}_8^{\rm odd}(\Gamma_2,\epsilon)$ 
is free with generators of weight
$(8,5)$, $(8,7)$, $(8,9)$, $(8,9)$, $(8,11)$, $(8,13)$, $(8,15)$, $(8,17)$ and $(8,23)$.
\end{theorem}
\begin{proof}
We use the covariants
$$
\begin{tiny}
\begin{aligned}
\xi_1=&160\,C_{1,6}C_{3,2}-208\,C_{2,0}C_{2,8}+250\,C_{2,4}^{2},\\
\xi_2=&60\,C_{1,6}C_{2,0}C_{3,2}+16\,C_{2,0}^{2}C_{2,8}-225\,C_{1,6}C_{5,2}-150\,C_{2,8}C_{4,0},\\
\xi_3^{(1)}=&4032\,C_{2,0}^{3}C_{2,8}+55800\,C_{1,6}C_{2,0}C_{5,2}-25000\,C_{1,6}C_{3,2}C_{4,0}-46125\,C_{2,0}
C_{2,4}C_{4,4},\\
&-159500\,C_{2,0}C_{3,2}C_{3,6}+17377500\,C_{1,6}C_{7,2}+90750\,C_{2,8}C_{6,0}+
675000\,C_{3,6}C_{5,2}-384375\,C_{4,4}^{2}\\
\xi_3^{(2)}=&112\,C_{1,6}C_{2,0}^{2}C_{3,2}-60\,C_{1,6}C_{2,0}C_{5,2}-150\,C_{1,6}C_{3,2}C_{4,0}-135\,
C_{2,0}C_{2,4}C_{4,4}-1440\,C_{2,0}C_{3,2}C_{3,6}\\
&+31500\,C_{1,6}C_{7,2}+450\,C_{2,8}C_{6,0}+5625\,C_{3,6}C_{5,2}-1125\,C_{4,4}^{2},\\
\xi_4=&1792\,C_{2,0}^{4}C_{2,8}+28750\,C_{1,6}C_{2,0}^{2}C_{5,2}-3685500\,C_{1,6}C_{2,0}C_{7,2}
-139200\,C_{1,6}C_{3,2}C_{6,0}\\
&-229650\,C_{1,6}C_{4,0}C_{5,2}-93600\,C_{2,0}C_{2,8}C_{6,0}-183150\,C_{2,0
}C_{3,6}C_{5,2}+166725\,C_{2,4}^{2}C_{6,0}\\
&-40500\,C_{2,4}C_{3,2}C_{5,2}-16875\,C_{2,4}C_{4,0}
C_{4,4}-72450\,C_{2,8}C_{4,0}^{2}+317700\,C_{3,2}^{2}C_{4,4}\\
&+256500\,C_{3,2}C_{3,6}C_{4,0}+38650500\,
C_{3,6}C_{7,2}+246600\,C_{5,4}^{2},\\
\xi_5=&807424\,C_{2,0}^{5}C_{2,8}-6707400000\,C_{1,6}C_{2,0}^{2}C_{7,2}
-1888920000\,C_{1,6}C_{2,0}C_{3,2}
C_{6,0}\\
&-785694375\,C_{1,6}C_{2,0}C_{4,0}C_{5,2}-278572500\,C_{1,6}C_{3,2}C_{4,0}^{2}
-120600000\,C_{2,0}^{2}C_{4,4}^{2}\\
&-42918750\,C_{2,0}C_{2,8}C_{4,0}^{2}+5193090000\,C_{2,0}C_{3,2}^{2}C_{4,4}
-271446918750\,C_{1,6}C_{4,0}C_{7,2}\\
&-5117321250\,C_{1,6}C_{5,2}C_{6,0}+338190300000\,C_{2,0}C_{3,6}C_{7,2}
+1145700000\,C_{2,0}C_{5,4}^{2}\\
&+62962200000\,C_{2,4}C_{3,2}C_{7,2}-450720000\,C_{2,4}C_{4,4}C_{6,0}
-1831612500\,C_{2,4}C_{5,2}^{2}\\
&+4053206250\,C_{2,8}C_{4,0}C_{6,0}-12202200000\,C_{3,2}C_{3,6}C_{6,0}
+20030895000\,C_{3,2}C_{4,4}C_{5,2}\\
&+6489787500\,C_{3,6}C_{4,0}C_{5,2}-8640074520000\,C_{2,8}C_{10,0}
-245226240000\,C_{4,6}C_{8,2}\\
&+170775360000\,C_{5,4}C_{7,4},\\
\xi_6=&8\,C_{2,0}C_{2,8}-25\,C_{2,4}^{2},\qquad
\xi_7=48\,C_{2,0}^{2}C_{2,8}-475\,C_{1,6}C_{5,2}+625\,C_{3,2}C_{3,6},\\
\xi_8=&2588867072\,C_{2,0}^{5}C_{2,8}-2215180800000\,C_{1,6}C_{2,0}^{2}C_{7,2}
+13431825000\,C_{1,6}C_{2,0}C_{4,0}C_{5,2}\\
&-97632787500\,C_{2,0}C_{2,8}C_{4,0}^{2}-125273250000\,C_{2,0}C_{3,2}^{2}C_{4,4}
+1345443750000\,C_{1,6}C_{4,0}C_{7,2}\\
&+7597800000000\,C_{2,0}C_{3,6}C_{7,2}+95399876250000\,C_{2,4}C_{3,2}C_{7,2}
-968719500000\,C_{2,4}C_{4,4}C_{6,0}\\
&-248030859375\,C_{2,4}C_{5,2}^{2}-178311712500\,C_{2,8}C_{4,0}C_{6,0}+
1077259500000\,C_{3,2}C_{4,4}C_{5,2}\\
&-143877610800000\,C_{2,8}C_{10,0}-5470416000000\,C_{4,6}C_{8,2}-25300674000000\,C_{5,4}C_{7,4}\,.
\end{aligned}
\end{tiny}
$$
We consider the following cusp forms:
\[
F_{8,5}=\frac{135}{8}\nu(\xi_1)\chi_5,
\quad
F_{8,7}=-\frac{405}{4}\nu(\xi_2)\chi_5,
\]
\[
F_{8,9}^{(1)}=\frac{675}{16}\nu(\xi_3^{(1)})\chi_5,
\quad
F_{8,9}^{(2)}=\frac{10125}{4}\nu(\xi_3^{(2)})\chi_5,
\]
\[
F_{8,11}=\frac{18225}{16}\nu(\xi_4)\chi_5
\quad
F_{8,13}=\frac{54675}{16}\nu(\xi_5)\chi_5,
\quad
F_{8,15}=-\frac{675}{4}\nu(\xi_6)\chi_5^3,
\]
\[
F_{8,17}=\frac{2025}{2}\nu(\xi_7)\chi_5^3,
\quad
F_{8,23}=-\frac{382725}{32}\nu(\xi_8)\chi_5^3\,.
\]

The Fourier expansion of
$$
W_{145}=F_{8,5} \wedge F_{8,7} \wedge F_{8,9}^{(1)} \wedge F_{8,9}^{(2)}
\wedge F_{8,11} \wedge F_{8,13} \wedge F_{8,15} \wedge F_{8,17} \wedge F_{8,23}
$$ 
starts with
$$
c \, (u^9-u^{7}-4u^{5}+4u^3+6u-6/u-4/u^3+4/u^5+1/u^7-1/u^9) X^{17}Y^{25}+\ldots
$$
with $c=-2^{17}\cdot 3^{10}\cdot 5^3\cdot 7 \cdot 59 \cdot 67 \cdot 103 \cdot 429$.
The order of vanishing of $W_{145}$ 
along $H_1$ is $5$, while along $H_4$ it is $4$, so $W_{145}$ is a multiple of 
$\chi_5^5\chi_{30}^4$ and a computation at the level of covariants gives
$$
W_{145}=
-2^{17}\cdot 3^{10}\cdot 5^3\cdot 7 \cdot 59 \cdot 67 \cdot 103 \cdot 429\,\chi_5^5\chi_{30}^4\,.
$$

\end{proof}

\begin{theorem}
The $R$-module $\mathcal{M}_8^{\rm ev}(\Gamma_2,\epsilon)$ 
is free with generators of weight
$(8,4)$, $(8,10)$, $(8,12)$, $(8,14)$, $(8,16)$,$(8,18)$, $(8,18)$, $(8,20)$ and $(8,22)$.
\end{theorem}
\begin{proof}
We use the following covariants
\begin{tiny}
\begin{align*}
\xi_1=&\, C_{3,8},\qquad \xi_5= C_{5,8}, \\
\xi_2=&\, 8\,C_{2,0}^{3}C_{3,8}-360\,C_{2,0}^{2}C_{5,8}-600\,C_{2,0}C_{3,2}C_{4,6}
+28000\,C_{1,6}C_{8,2}-1875\,C_{3,2}C_{6,6}^{(1)}+1500\,C_{3,2}C_{6,6}^{(2)}+3000\,C_{4,0}C_{5,8},\\
\xi_3=&\, 64\,C_{2,0}^{3}C_{5,8}+960\,C_{2,0}^{2}C_{3,2}C_{4,6}-26880\,C_{1,6}C_{2,0}C_{8,2}
-32760\,C_{2,0}C_{2,4}C_{7,4}-600\,C_{2,0}C_{4,0}C_{5,8}+405\,C_{3,8}C_{4,0}^{2}\\
&-974160\,C_{1,6}C_{10,2}+705600\,C_{2,4}C_{9,4}+267120\,C_{3,6}C_{8,2}-471240\,C_{4,4}C_{7,4}
+3263400\,C_{4,6}C_{7,2}-44280\,C_{5,2}C_{6,6}^{(1)}\\
&+41760\,C_{5,8}C_{6,0},\\
\xi_4=&\, -450785280\,C_{1,6}C_{2,0}C_{10,2}-209672400\,C_{1,6}C_{4,0}C_{8,2}
-107933000\,C_{2,0}C_{2,4}C_{9,4}+322793520\,C_{2,0}C_{3,6}C_{8,2}\\
&-93936640\,C_{2,0}C_{4,4}C_{7,4}+708825600\,C_{2,0}C_{4,6}C_{7,2}+27870759840\,C_{1,6}C_{12,2}-6460961760\,C_{3,6}C_{10,2}\\
&-10179070440\,C_{3,8}C_{10,0}-6501163200\,C_{4,4}C_{9,4}+2887120425\,C_{7,2}C_{6,6}^{(1)}
+4910108700\,C_{7,2}C_{6,6}^{(2)}\\
&-19333170\,C_{2,0}C_{5,2}C_{6,6}^{(1)}
+6700200\,C_{2,0}C_{5,2}C_{6,6}^{(2)}
+8466560\,C_{2,0}C_{5,8}C_{6,0}+104073340\,C_{2,4}C_{3,2}C_{8,2}\\
&+42245700\,C_{2,4}C_{4,0}C_{7,4}
+26659470\,C_{2,4}C_{5,4}C_{6,0}
-21600\,C_{4,0}^{2}C_{5,8}+1024\,C_{2,0}^{3}C_{3,2}C_{4,6}
+1024\,C_{2,0}^{5}C_{3,8},\\
\xi_6^{(1)}=&\, 8\,C_{2,0}C_{5,8}+25\,C_{2,4}C_{5,4}+30\,C_{3,2}C_{4,6},\qquad
\xi_6^{(2)}=\, C_{2,0}^{2}C_{3,8}-5\,C_{2,0}C_{5,8}-25\,C_{3,2}C_{4,6},\\
\xi_7=&\, 128\,C_{2,0}^{3}C_{3,8}+158200\,C_{1,6}C_{8,2}+214200\,C_{2,4}C_{7,4}
-88275\,C_{3,2}C_{6,6}^{(1)}+33900\,C_{3,2}C_{6,6}^{(2)}+39900\,C_{4,0}C_{5,8},\\
\xi_8=&\,768\,C_{2,0}^{4}C_{3,8}+2800000\,C_{1,6}C_{2,0}C_{8,2}-2782500\,C_{2,0}C_{2,4}C_{7,4}
-11979000\,C_{1,6}C_{10,2}+66990000\,C_{2,4}C_{9,4}\\
&-27636000\,C_{3,6}C_{8,2}
+30838500\,C_{4,4}C_{7,4}-117232500\,C_{4,6}C_{7,2}+880875\,C_{5,2}C_{6,6}^{(1)}-1039500\,C_{5,2}C_{6,6}^{(2)}\\
&-1342500\,C_{5,8}C_{6,0}.
\end{align*}
\end{tiny}
We consider the following cusp forms:
\[
F_{8,4}=-225\, \nu(\xi_1)\chi_5,
\quad
F_{8,10}=-\frac{6075}{512}\nu(\xi_2)\chi_5,
\quad
F_{8,12}=-\frac{6834375}{4}\nu(\xi_3)\chi_5,
\]
\[
F_{8,14}=\frac{102515625}{256}\nu(\xi_4)\chi_5,
\quad
F_{8,16}=50625\nu(\xi_5)\chi_5^3
\quad
F_{8,18}^{(1)}=\frac{151875}{4}\nu(\xi_6^{(1)})\chi_5^3,
\]
\[
F_{8,18}^{(2)}=-\frac{6075}{16}\nu(\xi_6^{(2)})\chi_5^3,
\quad
F_{8,20}=\frac{151875}{32}\nu(\xi_7)\chi_5^3,
\quad
F_{8,22}=-\frac{1366875}{16}\nu(\xi_8)\chi_5^3\,.
\]
Then
\[
W_{170}=
F_{8,4} \wedge F_{8,10} \wedge F_{8,12}  \wedge F_{8,14} \wedge F_{8,16} 
\wedge F_{8,18}^{(1)} \wedge F_{8,18}^{(2)} 
\wedge F_{8,20} \wedge F_{8,22} 
\]
is a cusp form in $S_{170}(\Gamma_2,\epsilon)$ and its Fourier expansion starts with
\[
2^{36}\cdot 3^{13}\cdot 5^8\cdot 7^3 \cdot 19\,   
(u^9+u^{7}-4u^{5}-4u^3+6u+6/u-4/u^3-4/u^5+1/u^7+1/u^9) X^{19}Y^{29}+\ldots
\]
One can check that the order of vanishing of  $W_{170}$
along $H_1$ is $4$ while along $H_4$ it is $5$, so $W_{170}$ is a 
multiple of $\chi_{5}^4\chi_{30}^5$. A calculation with the covariants shows
$$
W_{170}=2^{36}\cdot 3^{13}\cdot 5^8\cdot 7^3 \cdot 19 \, \chi_{5}^4\chi_{30}^5\,.
$$
\end{proof}
\end{section}
\begin{section}{The case $j=10$}

\begin{theorem}
The $R$-module $\mathcal{M}_{10}^{\rm odd}(\Gamma_2,\epsilon)$ 
is free with generators of weight
$(10,5)$, $(10,7)$, $(10,9)$, $(10,9)$, $(10,11)$, $(10,11)$, $(10,13)$, $(10,13)$,
 $(8,15)$, $(10,15)$ and $(10,17)$.
\end{theorem}
\begin{theorem}
The $R$-module $\mathcal{M}_{10}^{\rm ev}(\Gamma_2,\epsilon)$ 
is free with generators of weight
$(10,8)$, $(10,10)$, $(10,10)$, $(10,12)$, $(10,12)$, $(10,14)$,$(10,14)$, $(10,16)$,
$(10,16)$, $(10,18)$ and $(10,20)$.
\end{theorem}

The proofs in both cases are similar to the cases above. The covariants used are quite
big and we refer for these to  \cite{BFvdG}.
\end{section}
\begin{section}{The character $\epsilon$ of $\Gamma_2$}
Maa{\ss} showed in \cite{Maass} that the abelianization of $\Gamma_2$ is isomorphic to
${\ZZ}/2{\ZZ}$. So $\Gamma_2$ has one non-trivial character $\epsilon$ and it is
of order $2$. It can be described as the composition 
$$
{\rm Sp}(4,{\ZZ})\xrightarrow{ \bmod \, 2 }
{\rm Sp}(4,{\ZZ}/2{\ZZ}) {\buildrel \cong \over \longrightarrow}
\mathfrak{S}_6  \xrightarrow{\rm sign} 
\{ \pm 1 \}\, .
$$
The following rules may help in easily determining the value $\epsilon(\gamma)$. If 
$$
\gamma= \left( \begin{matrix} a & b \\ c & d \\ \end{matrix} \right)
$$
then one has
$$
\epsilon(\left( \begin{matrix} a & b \\ c & d \\ \end{matrix} \right))
= 
\epsilon(\left( \begin{matrix} c & d \\ a & b \\ \end{matrix} \right))
=
\epsilon(\left( \begin{matrix} b & a \\ d & c \\ \end{matrix} \right))
=
\epsilon(\left( \begin{matrix} d & c \\ b & a \\ \end{matrix} \right))
$$
as one sees by applying $J=(0,1_g;-1_g,0)$ on the left and/or on the right.

If $\gamma$ satisfies
$$
\det(a)\equiv \det(b) \equiv \det(c) \equiv \det(d) \equiv 0 \, (\bmod \, 2)
$$
then we have $\epsilon(\gamma)=-\epsilon(\gamma_0)$ with $\gamma_0$
obtained from $\gamma$ by replacing the first row by minus the third row
and the third row by the first row. For this matrix $\gamma_0$ at least one
of $\det(a_0)$, $\det(b_0)$, $\det(c_0)$, $\det(d_0)$ is not zero modulo $2$.

Using this we arrive at the case where $\gamma$ has the property that 
$\det(c)\not\equiv 0 (\bmod \, 2)$.
\begin{proposition}
For $\gamma= (a,b;c,d) \in \Gamma_2$ with $\det(c)\not\equiv 0 (\bmod \, 2)$ 
we have $\epsilon(\gamma)=(-1)^{\rho}$
with $\rho$ given by
$$
a_1c_1 +a_2c_1 +a_2c_2 +a_3c_3 +a_4c_3 +a_4c_4 +c_1c_2 +c_2c_3 +
c_3c_4 +c_1d_4 +c_2d_3 +c_2d_4 +c_3d_2 +c_4d_1 +c_4d_2
$$
where the $2\times 2$ matrices are written as 
$\left( \begin{smallmatrix} x_1 & x_2 \\ x_3 & x_4 \\ \end{smallmatrix} \right)$.
\end{proposition}
The proof is omitted. 
\end{section}

\end{document}